\documentclass[a4paper,12pt]{article}
\usepackage[utf8]{inputenc}
\usepackage{graphicx}

\usepackage{mathrsfs}
\usepackage{enumerate}
\usepackage{multicol}
\usepackage{color}
\usepackage{amsmath,amssymb,amscd}
\usepackage{amsthm}

\usepackage{pdfpages}
\usepackage{hyperref}

\usepackage{times}
\usepackage[varg]{txfonts}

\usepackage[top=1.5cm, bottom=1.5cm, left=1cm, right=1cm]{geometry}

\theoremstyle{plain}
\newtheorem{thm}{Theorem}

\newtheorem{cor}{Corollary}
\newtheorem{prop}{Proposition}

\theoremstyle{definition}
\newtheorem{dfn}{Definition}

\theoremstyle{remark}
\newtheorem*{rem}{Remark}
\newtheorem*{opp}{Open Problems}

\title{Remarks on the notion of homo-derivations}
\author{Eszter Gselmann and Gergely Kiss}

\begin{document}

\maketitle

\begin{abstract}
 The purpose of this paper is to study the (different) notions of homo-derivations. These are additive 
 mappings $f$ of a ring $R$ that also fulfill the identity 
 \[
  f(xy)=f(x)y+xf(y)+f(x)f(y)
  \qquad 
  \left(x, y\in R\right), 
 \]
 or (in case of the other notion) the system of equations 
 \[
  \begin{array}{rcl}
   f(xy)&=&f(x)f(y)\\
   f(xy)&=&f(x)y+xf(y)
  \end{array}
  \qquad 
  \left(x, y\in R\right). 
 \]
Our primary aim is to investigate the above equations without additivity as well as 
the following Pexiderized equation
\[
 f(xy)=h(x)h(y)+xk(y)+k(x)y. 
\]
The obtained results show that under rather mild assumptions homo-derivations can be fully characterized, even without the additivity assumption.  
\end{abstract}

\begin{center}
 \emph{\small{Dedicated to the $70$\textsuperscript{th} birthday of Professor Antal Járai. }}
\end{center}

\section{Introduction and preliminaries}

The main aim of this paper is to present some characterization theorems concerning homomorphisms, derivations and also homo-derivations. 
Thus, at first, we list some notions and preliminary results that will be used in the sequel.
All of these statements and definitions can be found in Kuczma \cite{Kuc09}\index{Kuczma, M.} and
in Zariski--Samuel\index{Zariski, O.}\index{Samuel, P.} \cite{ZarSam75} and also in Kharchenko\index{Kharchenko, V. L.} \cite{Kha91}.

\subsection*{Homomorphisms and derivations}

\begin{dfn}
 Let $P$ and $Q$ be (not necessarily unital) rings. A function $f\colon P\to Q$ is called a 
 \emph{homomorphism} (of $P$ into $Q$) if it is additive, i.e. 
 \[
f(x+y)=f(x)+f(y)
\quad
\left(x, y\in P\right)
\]
and also 
\[
 f(xy)=f(x)f(y) 
 \qquad 
 \left(x, y\in P\right)
\]
holds. 

If moreover, $f$ is one-to-one, then $f$ is called a \emph{monomorphism}. 
If $f$ is onto, then $f$ is called an \emph{epimorphism}. 
A homomorphism which is a monomorphism and an epimorphism is called an \emph{isomorphism}. 
In case $P=Q$, the function $f$ is termed to be an \emph{endomorphism}. 
\end{dfn}

\begin{dfn}
Let $Q$ be a (not necessarily unital) ring and let $P$ be a subring of $Q$.
A function $f\colon P\rightarrow Q$ is called a \emph{derivation} if it is additive 
and also satisfies the so-called \emph{Leibniz rule}, i.e.  equation
\[
f(xy)=f(x)y+xf(y)
\quad
\left(x, y\in P\right). 
\]
\end{dfn}

\begin{rem}
 Let $Q$ be a ring and let $P$ be a subring of $Q$.
Functions $f\colon P\rightarrow Q$ fulfilling the Leibniz rule only, will be termed \emph{Leibniz functions}. 
\end{rem}

Among derivations one can single out so-called inner derivations, similarly as in the case of automorphisms. 

\begin{dfn}
 Let $R$ be a ring and $b\in R$, then the mapping $\mathrm{ad}_{b}\colon R\to R$ defined by 
 \[
  \mathrm{ad}_{b}(x)=\left[x, b\right]=xb-bx
  \qquad 
  \left(x\in R\right)
 \]
is a derivation. A derivation $f\colon R\to R$ is termed to be an \emph{inner derivation} if there is a 
$b\in R$ so that $f=\mathrm{ad}_{b}$. We say that a derivation is an \emph{outer derivation} if it is not inner. 
\end{dfn}

An another fundamental example for derivations is the following.

\begin{rem}
Let $\mathbb{F}$ be a field, and let in the above definition $P=Q=\mathbb{F}[x]$
be the ring of polynomials with coefficients from $\mathbb{F}$. For a polynomial
$p\in\mathbb{F}[x]$, $p(x)=\sum_{k=0}^{n}a_{k}x^{k}$, define the function
$f\colon \mathbb{F}[x]\rightarrow\mathbb{F}[x]$ as
\[
f(p)=p',
\]
where $p'(x)=\sum_{k=1}^{n}ka_{k}x^{k-1}$ is the derivative of the polynomial $p$.
Then the function $f$ clearly fulfills
\[
\begin{array}{rcl}
f(p+q)&=&f(p)+f(q)\\
f(pq)&=&pf(q)+qf(p)
\end{array}
\]
for all $p, q\in\mathbb{F}[x]$. Hence $f$ is a derivation.
\end{rem}

Clearly, commutative rings admit only \emph{trivial} inner derivations. At the same time, it 
not so evident whether commutative rings (or fields) \emph{do} or \emph{do not} admit nontrivial 
outer derivations. To answer this problem partially, here we recall Theorem 14.2.1 from Kuczma \cite{Kuc09}.

\begin{thm}\label{T14.2.1}
Let $\mathbb{K}$ be a field of characteristic zero, let $\mathbb{F}$ be a subfield of $\mathbb{K}$, 
let $S$ be an algebraic base of $\mathbb{K}$ over $\mathbb{F}$,
if it exists, and let $S=\emptyset$ otherwise.
Let $f\colon \mathbb{F}\to \mathbb{K}$ be a derivation.
Then, for every function $u\colon S\to \mathbb{K}$,
there exists a unique derivation $g\colon \mathbb{K}\to \mathbb{K}$
such that $g \vert_{\mathbb{F}}=f$ and $g \vert_{S}=u$.
\end{thm}

In \cite{Sof00}, El Sofy introduced the notion of homo-derivations. 
After that several results appeared where the authors proved commutativity results 
for the domain of these mappings, see e.~g. \cite{AlhMut19a, AlhMut19, MahAlaNaj19, RehAlRadMut19}. 
It is objectionable however that there have not been made attempt to characterize or to compare these notions. 
One of the main purpose of this work is to clarify these problems.

\begin{dfn}[El Sofy \cite{Sof00}]
Let $Q$ be a ring and let $P$ be a subring of $Q$.
A function $f\colon P\rightarrow Q$ is called a \emph{homo-derivation} if it is additive 
and also satisfies the equation
\[
f(xy)=f(x)y+xf(y)+f(x)f(y)
\quad
\left(x, y\in P\right). 
\]
\end{dfn}

We remark that there can be found some other ways of introducing homo-derivations.  
Here we present a further definition as it appears in \cite{MehPaj03}.

\begin{dfn}[Mehdi Ebrahimi--Pajoohesh \cite{MehPaj03}]
Let $Q$ be a ring and let $P$ be a subring of $Q$.
A function $f\colon P\rightarrow Q$ is called a \emph{homo-derivation} if it is a homomorphism  
and also satisfies the Leibniz rule. 
\end{dfn}

\subsection*{Polynomials and the Levi-Civit\`{a} equation}

As we will see in the second section, the notion of exponential polynomials and the 
so-called Levi-Civit\`{a} functional equation will play a distinguished role while 
proving our results.

In view of Theorem 10.1 of \cite{Sze91}, if $(G, \cdot)$ is an Abelian group, then
any function $f\colon G\to \mathbb{C}$ satisfying the so-called
\emph{Levi-Civit\`{a} functional equation}, that is,
\begin{equation}\label{Eq_levi}
 f(x\cdot y)= \sum_{i=1}^{n}g_{i}(x)h_{i}(y)
 \qquad
 \left(x, y\in G\right)
\end{equation}
for some positive integer $n$ and functions 
$g_{i}, h_{i}\colon G\to\mathbb{C}\; (i=1, \ldots, n)$, is an exponential polynomial of
order at most $n$. 

At the same time, in equation \eqref{Eq_levi} not only the function $f$, but also the 
mappings $g_{i}, h_{i}\colon G\to\mathbb{C}\; (i=1, \ldots, n)$ are assumed to be unknown.

If either the functions $g_{1}, \ldots, g_{n}$ or 
the functions $h_{1}, \ldots, h_{n}$ are \emph{linearly dependent}, 
then the number $n$ appearing in equation \eqref{Eq_levi} can be reduced 
and in this case the general solution of the equation can contain arbitrary functions, 
we shall call this case \emph{degenerate}.

Alternatively, if the functions $h_{1}, \ldots, h_{n}$ are \emph{linearly independent}, 
then $g_{1}, \ldots, g_{n}$ are linear combinations of 
the translates of $f$, hence they are exponential  polynomials of
order at most $n$, too. Moreover, they are built up from the same
additive and exponential functions as the function $f$. Roughly speaking this is 
Theorem 10.4 of \cite{Sze91} which is the following.

\begin{thm}\label{Szekely}
 Let $G$ be an Abelian group, $n$ be a positive integer and $f, g_{i}, h_{i}\colon G\to \mathbb{C}\; (i=1, \ldots, n)$ be functions so that both the sets
 $\left\{g_{1}, \ldots, g_{n}\right\}$ and $\left\{h_{1}, \ldots, h_{n}\right\}$ are linearly independent.
 The functions  $f, g_{i}, h_{i}\colon G\to \mathbb{C}\; (i=1, \ldots, n)$ form a \emph{non-degenerate} solution of equation \eqref{Eq_levi}
 if and only if
 \begin{enumerate}[(a)]
  \item there exist positive integers $k, n_{1}, \ldots, n_{k}$ with $n_{1}+\cdots+n_{k}=n$;
  \item there exist different nonzero complex exponentials $m_{1}, \ldots, m_{k}$;
  \item for all $j=1, \ldots, k$ there exists linearly independent sets of complex additive functions \[\left\{a_{j, 1}, \ldots,  a_{j, n_{j}-1}\right\};\]
  \item there exist polynomials $P_{j}, Q_{i, j}, R_{i, j}\colon \mathbb{C}^{n_{j}-1}\to \mathbb{C}$ for all $i=1, \ldots, n; j=1, \ldots, k$ in
  $n_{j}-1$ complex variables and of degree at most $n_{j}-1$;
 \end{enumerate}
 so that we have
 \[
  f(x)= \sum_{j=1}^{k}P_{j}\left(a_{j, 1}(x), \ldots, a_{j, n_{j}-1}(x)\right)m_{j}(x)
 \]
 \[
  g_{i}(x)= \sum_{j=1}^{k}Q_{i, j}\left(a_{j, 1}(x), \ldots, a_{j, n_{j}-1}(x)\right)m_{j}(x)
 \]
 and
 \[
  h_{i}(x)= \sum_{j=1}^{k}R_{i, j}\left(a_{j, 1}(x), \ldots, a_{j, n_{j}-1}(x)\right)m_{j}(x)
 \]
for all $x\in G$ and  $i=1, \ldots, n$. 
\end{thm}

Let $G$ be a groupoid  and $\mathbb{F}$ be field. 
Given $(A_{i, j})_{i, j}\in \mathscr{M}_{n\times n}(\mathbb{F})$ and $(\Gamma^{(k)}_{i, j})_{i, j} \in \mathscr{M}_{n\times n}(\mathbb{F})$, 
in \cite{McK77} McKiernan studied the following problems. 
\begin{enumerate}[(1)]
 \item Find all functions $h,f_{i}, g_{i}\colon G\to \mathbb{F}$ ($i= 1, \ldots, n$)  satisfying the equation 
 \begin{equation}\label{Eq_McKiernan1}
 h(xy)=\sum_{i, j=1}^{n} A_{i, j} f_{i}(x)g_{j}(y)
 \qquad 
 \left(x, y\in G\right). 
 \end{equation}
  \item Find all functions $g_{i}\colon G\to \mathbb{F}$ ($i=1, \ldots, n$) 
 satisfying the system of equations 
 \begin{equation}\label{Eq_McKiernan2}
 g_{k}(xy)=\sum_{i, j=1}^{n}\Gamma_{i, j}^{(k)}g_{i}(x)g_{j}(y)
 \qquad 
 \left(
x, y\in G, k=1, \ldots, n 
 \right)
\end{equation}
\end{enumerate}

The solutions are obtained by showing that the two problems are essentially equivalent, 
then transforming to a matrix problem and applying one of his earlier results concerning a multiplicative matrix equation. 
All in all, the main result of \cite{McK77} is that if 
\begin{enumerate}[(a)]
 \item $G$ is a (not necessarily commutative) semigroup, 
 \item $\mathbb{F}$ is an algebraically closed field with $\mathrm{char}(F)\geq (n-1)!$, 
 \item $\det \left(A_{i, j}\right)\neq 0$, 
 \item $f_{1}, \ldots, f_{n}$ are linearly independent, 
 \item $g_{1}, \ldots, g_{n}$ are linearly independent, 
\end{enumerate}
then the solutions of equation \eqref{Eq_McKiernan1} as well as \eqref{Eq_McKiernan2} are exponential polynomials.

In what follows \cite[Lemma 4.2]{Sze91}, that is, the statement below will be utilized several times. 

\begin{thm}
 Let $G$ be an Abelian group, $\mathbb{K}$ a field, 
 $X$ a $\mathbb{K}$-linear space and $\mathscr{V}$ a translation invariant linear space 
 of $X$-valued functions on $G$. 
 Let $k_{i}$ be nonnegative integers, $n\geq 1$, 
 $m_{i}\colon G\to \mathbb{K}$ different nonzero exponentials, 
 $A_{i}\colon G^{n_{i}}\to X$  symmetric, $k_{i}$-additive functions and 
 $q_{i}\colon G\to X$ polynomials of degree at most $k_{i}-1$ ($i=1, \ldots, n$). 
 If the function 
 \[
 \sum_{i=1}^{n}\left(\mathrm{diag}(A_{i})+q_{i}\right)m_{i}
 \]
 belongs to 
 $\mathscr{V}$, then there exist polynomials $r_{i}\colon G\to X$ of degree at most 
 $k_{i}-1$ such that $\left(\mathrm{diag}(A_{i})+r_{i}\right)m_{i}$ belongs to $\mathscr{V}$ for 
 $i=1, \ldots, n$. 
\end{thm}

From this, with the choice $\mathscr{V}= \left\{0\right\}$ and $X=\mathbb{K}$ we get the following. 

\begin{prop}
  Let $G$ be an Abelian group, $\mathbb{K}$ be a field and $n$ be a positive integer. 
  Suppose that for each $x\in G$ 
  \[
   p_{1}(x)m_{1}(x)+\cdots+p_{n}(x)m_{n}(x)=0
  \]
holds, where $m_{1}, \ldots, m_{n}\colon G\to \mathbb{K}$ are different exponentials and 
$p_{1}, \ldots, p_{n}\colon G\to \mathbb{K}$ are (generalized) polynomials. 
Then for all $i=1, \ldots, n$ the polynomial $p_{i}$ is identically zero. 
\end{prop}

\section{Results}

\subsection*{The functional equation of homo-derivations}

As the theorem below shows, the notion of homo-derivations (in the sense of El Sofy \cite{Sof00}) can 
be characterized even \emph{without} additivity supposition.

\begin{thm}\label{Thm4}
Let $P$ and $Q$ be rings such that $P$ is a subring of $Q$ and assume that 
$\varepsilon$ is an arbitrary nonzero  element of the center of $Q$. 
 Function $h\colon P\to Q$ fulfills the functional equation 
 \begin{equation}\label{Eq_homo-deriv}
  h(xy)= h(x)y+xh(y)+\varepsilon h(x)h(y)
 \end{equation}
for all $x, y\in P$ if and only if there exists a multiplicative function $m\colon P\to Q$ such that 
\[
\varepsilon h(x)= m(x)-x 
\qquad 
\left(x\in P\right). 
\]
\end{thm}

\begin{proof}
 Multiplying equation \eqref{Eq_homo-deriv} with $\varepsilon$ leads to 
 \[
 \varepsilon h(xy)= \varepsilon h(x)y+x \varepsilon h(y)+\varepsilon h(x)\varepsilon h(y) 
  \qquad 
  \left(x, y\in P\right). 
 \]
If we add $xy$ to both sides of this equation, then 
  \[
 \varepsilon h(xy)+xy= \varepsilon h(x)y+x \varepsilon h(y)+\varepsilon h(x)\varepsilon h(y) +xy
  \qquad 
  \left(x, y\in P\right), 
 \]
 follows. Observe however that 
 \[
  \varepsilon h(x)y+x \varepsilon h(y)+\varepsilon h(x)\varepsilon h(y) +xy= 
  (\varepsilon h(x)+x)\cdot \left(\varepsilon h(y)+y\right) 
  \qquad 
  \left(x, y\in P\right), 
 \]
yielding exactly that the mapping $m\colon P\to Q$ defined through 
\[
 m(x)= \varepsilon h(x)+x 
 \qquad 
 \left(x\in P\right)
\]
is multiplicative. 
\end{proof}

\begin{rem}
 In the case the ring $Q$ is unital, $\varepsilon=1$ is a nonzero central element of the ring $Q$. Hence 
 any solution of the above equation can be represented as 
 \[
  h(x)= m(x)-x 
  \qquad 
  \left(x\in P\right), 
 \]
with an appropriate multiplicative function $m$. 
\end{rem}

As an immediate corollary we get the following. 

\begin{cor}
 Let $P$ be a subring of the ring $Q$ and assume that  $\varepsilon$ is an arbitrary nonzero element of the center of $Q$. 
 The  additive mapping $a\colon P\to Q$ fulfills functional equation 
 \[
  a(xy)= a(x)y+xa(y)+\varepsilon a(x)a(y)
 \]
for all $x, y\in P$ if and only if there exists a homomorphism $\varphi\colon P\to Q$ such that 
\[
 \varepsilon a(x)=\varphi(x)-x
 \qquad 
 \left(x\in P\right). 
\]
\end{cor}

In view of the above result, homo-derivations in the sense of El Sofy \cite{Sof00} on commutative rings can be characterized. 

\begin{cor}
Let $P$ be a subring of the commutative ring $Q$ and $a\colon P\to Q$ be a homo-derivation in the sense of El Sofy \cite{Sof00}. 
 Then and only then 
 there exists a homomorphism $\varphi\colon P\to Q$ such that 
\[
 a(x)=\varphi(x)-x
 \qquad 
 \left(x\in P\right). 
\]
\end{cor}

The proposition below considers the other notion of homo-derivations. 

\begin{prop}\label{Prop1}
Let $P$ be a subring of the ring $Q$, the 
function $f\colon P\to Q$ fulfills the system of equations
 \[
 \begin{array}{rcl}
  f(xy)&=&f(x)f(y)\\
  f(xy)&=&f(x)y+xf(y)
 \end{array}
 \qquad 
 \left(x\in \mathbb{F}\right)
 \]
if and only if there exists a non-zero constant $\alpha\in Q$ such that  $\alpha \cdot f$ and 
$f\cdot \alpha$ are identically zero.  
\end{prop}

\begin{proof}
 Assume that the function $f\colon P\to Q$ satisfies the above system. 
 Then clearly, functional equation 
 \[
  f(x)f(y)=f(x)y+xf(y) 
  \qquad 
  \left(x, y\in P\right)
 \]
also holds. 
After some rearrangement  we arrive at 
\[
 f(x)\left[f(y)-y\right]= xf(y)
 \qquad 
 \left(x\in P\right). 
\]
Since $f$ must be a Leibniz function, 
\[
 f(y)-y=0 
\]
cannot hold for all $y\in P$. Thus there exists $y^{*}\in P$ such that 
$ f(y^{*})-y^{*}\neq 0$, from which 
\[
 f(x)\left[f(y^{*})-y^{*}\right]= xf(y^{*})
 \qquad 
  \left(x\in P\right). 
\]
In other words, there are constants $\alpha, \beta \in Q$ with $\alpha\neq 0$ such that 
\[
 f(x)\alpha = x \beta 
 \qquad 
 \left(x\in P\right). 
\]
Writing this back to the Leibniz equation $\beta =0$ follows, yielding that $f\cdot \alpha$ is identically zero. 
Changing the role of $x$ and $y$ in the above argument, $\alpha \cdot f \equiv 0$ also follows. 
\end{proof}

\begin{cor}
Let $P$ be a subring of the ring $Q$, 
a function $a\colon P\to Q$ is a homo-derivation in the sense of 
 Mehdi Ebrahimi--Pajoohesh \cite{MehPaj03} if and only if there exists a nonzero constant $\alpha\in Q$ such that 
 $\alpha \cdot a\equiv a\cdot \alpha\equiv 0$. Especially, if $Q$ has no zero-divisors, then $a$ has to be identically zero. 
\end{cor}

\subsection*{On the functional equation $f(xy)=h(x)h(y)+xk(y)+k(x)y$}

In this subsection we determine the solutions $f, h, k\colon \mathbb{F}\to \mathbb{K}$ of the functional equation 
\begin{equation}\label{Eq_gen}
 f(xy)=h(x)h(y)+xk(y)+k(x)y 
 \qquad 
 \left(x, y\in \mathbb{F}\right). 
\end{equation}
From this, the additive solutions of the same equation will follow immediately.  
Here we suppose that 
\emph{$\mathbb{K}$ is an algebraically closed field with $\mathrm{char}(\mathbb{K})\neq 2$ and 
$\mathbb{F}$ is a subfield of $\mathbb{K}$}. 

Observe that equation \eqref{Eq_gen} is a special Levi-Civit\`{a} equation. 
Therefore according to the value $\dim \mathrm{lin}\left(\mathrm{id}, h, k\right)$, we have to distinguish several cases. 
Clearly, \[\dim \mathrm{lin}\left(\mathrm{id}, h, k\right)=3\] means that the mappings involved in the right hand side of 
\eqref{Eq_gen} are linearly independent. Thus in the degenerate cases we have $\dim \mathrm{lin}\left(\mathrm{id}, h, k\right)<3$. 

\subsubsection*{Degenerate cases}

Firstly, let us assume that $\dim \mathrm{lin}\left(\mathrm{id}, h, k\right)=1$. In this situation there exist 
constants $\lambda_{1}, \lambda_{2}\in \mathbb{K}$ such that 
\[
 h(x)=\lambda_{1}x 
 \quad 
 \text{and}
 \quad 
 k(x)=\lambda_{2}x 
 \qquad 
 \left(x\in \mathbb{F}\right). 
\]

\begin{prop} 
Let $\lambda_{1}, \lambda_{2}\in \mathbb{K}$ be arbitrarily fixed. 
 Function $f\colon \mathbb{F}\to \mathbb{K}$ fulfills functional equation \eqref{Eq_gen} 
 if and only if 
 \[
  f(x)= \left(\lambda_{1}^{2}+2\lambda_{2}\right)x 
  \qquad 
  \left(x\in \mathbb{F}\right). 
 \]
\end{prop}
\begin{proof}
 In case 
 \[
 h(x)=\lambda_{1}x 
 \quad 
 \text{and}
 \quad 
 k(x)=\lambda_{2}x 
 \qquad 
 \left(x\in \mathbb{F}\right), 
\]
our equation reduces to 
\[
 f(xy)= \left(\lambda_{1}^{2}+2\lambda_{2}\right)xy 
  \qquad 
  \left(x, y\in \mathbb{F}\right), 
\]
from which the results follows immediately. 
\end{proof}

Secondly, assume that $\dim \mathrm{lin}\left(\mathrm{id}, h, k\right)=2$, which can happen in different ways. 
If $\left\{ \mathrm{id}, h\right\}$ are linearly dependent, that is 
\[
 h(x)= \lambda x 
 \qquad 
 \left(x\in \mathbb{F}\right)
\]
holds with a certain $\lambda\in \mathbb{K}$, 
then we have the following. 

\begin{prop}\label{Prop4}
 Let $\lambda\in \mathbb{K}$ be an arbitrary constant. 
 Functions $f, k\colon \mathbb{F}\to \mathbb{K}$ fulfill the functional equation 
 \[
  f(xy)=\lambda^{2}xy+xk(y)+k(x)y  
  \qquad 
  \left(x, y\in \mathbb{F}\right)
 \]
if and only if there exists a Leibniz function 
$\delta \colon \mathbb{F}\to \mathbb{K}$ such that 
\[
 \begin{array}{rcl}
k(x)&=&k(1)x+\delta(x)\\
f(x)&=&\left(\lambda^{2}+ 2k(1)\right)x+\delta(x)
\end{array}
\qquad 
\left(x\in \mathbb{F}\right). 
\]
\end{prop}

\begin{proof}
Under the above assumptions, equation \eqref{Eq_gen} with $y=1$ yields that 
\[
 f(x)= (\lambda^{2}+k(1))x+k(x) 
 \qquad 
 \left(x\in \mathbb{F}\right). 
\]
Writing this back into our equation, we get that 
\[
 k(xy)+(\lambda^{2}+k(1))xy= xk(y)+k(x)y +\lambda^{2}xy
 \qquad 
 \left(x, y\in \mathbb{F}\right).  
\]
This means that the function $k\colon \mathbb{F}\to \mathbb{K}$ fulfills 
\[
 k(xy)+k(1)xy= xk(y)+k(x)y 
 \qquad 
 \left(x, y\in \mathbb{F}\right). 
\]
Thus, there exists a Leibniz function $\delta\colon \mathbb{F}\to \mathbb{K}$ such that 
\[
\begin{array}{rcl}
k(x)&=&k(1)x+\delta(x)\\
f(x)&=&\left(\lambda^{2}+2k(1)\right)x+\delta(x)
\end{array}
\qquad 
\left(x\in \mathbb{F}\right). 
\]
\end{proof}

\begin{cor}
 Let $\lambda\in \mathbb{K}$ be an arbitrary constant. 
 The additive functions $f, k\colon \mathbb{F}\to \mathbb{K}$ fulfill the functional equation 
 \[
  f(xy)=\lambda^{2}xy+xk(y)+k(x)y  
  \qquad 
  \left(x, y\in \mathbb{F}\right)
 \]
if and only if there exists a derivation 
$d \colon \mathbb{F}\to \mathbb{K}$ such that 
\[
 \begin{array}{rcl}
k(x)&=&k(1)x+d(x)\\
f(x)&=&\left(\lambda^{2}+ 2k(1)\right)x+d(x)
\end{array}
\qquad 
\left(x\in \mathbb{F}\right). 
\]
\end{cor}

Our second case is when $\left\{ \mathrm{id}, k\right\}$ are linearly dependent, that is if 
\[
 k(x)= \lambda x 
 \qquad 
 \left(x\in \mathbb{F}\right)
\]
with a certain $\lambda \in \mathbb{K}$. 

\begin{prop}\label{Prop3}
 Let $\lambda\in \mathbb{F}$ be arbitrarily fixed.  
 Functions $f, h\colon \mathbb{F}\to \mathbb{K}$ fulfill the functional equation 
 \[
  f(xy)=h(x)h(y)+2\lambda xy  
  \qquad 
  \left(x, y\in \mathbb{F}\right)
 \]
if and only if there exists a multiplicative function $m\colon \mathbb{F}\to \mathbb{K}$ such that 
\[
 \begin{array}{rcl}
  h(x)&=&h(1)m(x)\\
  f(x)&=&h(1)^{2}m(x)+2\lambda x
 \end{array}
\qquad 
\left(x\in \mathbb{F}\right). 
\]
\end{prop}

\begin{proof}
 Define the function $\widetilde{f}\colon \mathbb{F}\to \mathbb{K}$ through 
 \[
  \widetilde{f}(x)=f(x)-2\lambda x 
  \qquad 
  \left(x\in \mathbb{F}\right)
 \]
to deduce that 
\[
 \widetilde{f}(xy)= h(x)h(y)
 \qquad 
 \left(x, y\in \mathbb{F}\right). 
\]
This identity with $y=1$ implies that 
\[
 \widetilde{f}(x)=h(1)h(x)
 \qquad 
 \left(x\in \mathbb{F}\right). 
\]
Therefore, 
\begin{enumerate}[A)]
\item either $h(1)=0$ from which 
\[
 f(x)=2\lambda x 
 \qquad 
 \left(x\in \mathbb{F}\right)
\]
and $h\equiv 0$ follows; 
\item or $h(1)\neq 0$ from which we get that 
\[
 \dfrac{h(xy)}{h(1)}= \dfrac{h(x)}{h(1)}\cdot \dfrac{h(y)}{h(1)} 
 \qquad 
 \left(x, y\in \mathbb{F}\right). 
\]
\end{enumerate}
All in all, there exists a multiplicative function $m\colon \mathbb{F}\to \mathbb{K}$ such that 
\[
 h(x)=h(1)m(x)
 \qquad 
 \left(x\in \mathbb{F}\right). 
\]
From this we also obtain that 
\[
 f(x)=h(1)^{2}m(x)+2\lambda x
 \qquad 
 \left(x\in \mathbb{F}\right). 
\]

\end{proof}

\begin{rem}
 Similarly, as previously, the additive solutions of the above equation are of the form 
 \[
 \begin{array}{rcl}
  h(x)&=&h(1)\varphi(x)\\
  f(x)&=&h(1)^{2}\varphi(x)+2\lambda x
 \end{array}
\qquad 
\left(x\in \mathbb{F}\right),  
\]
with a certain homomorphism $\varphi\colon \mathbb{F}\to \mathbb{K}$. 
\end{rem}

Finally, the last possibility is that $\left\{h, k\right\}$ is linearly dependent. In this case 
there are constants  $\lambda_{1}, \lambda_{2}\in \mathbb{K}$ not vanishing simultaneously such that 
\[
 \lambda_{1}k(x)+\lambda_{2}h(x)=0 
\]
for all $x\in \mathbb{F}$. Again we have the following alternative. 
\begin{enumerate}[A)]
 \item Either $\lambda_{2}=0$ and then $k\equiv 0$. In this case \eqref{Eq_gen} is 
 \[
  f(xy)= h(x)h(y) 
  \qquad 
  \left(x, y\in \mathbb{F}\right). 
 \]
Using Proposition \ref{Prop3} we finally get that there exists a multiplicative function 
$m\colon \mathbb{F}\to \mathbb{K}$ such that 
\[
 \begin{array}{rcl}
  h(x)&=&h(1)m(x)\\
  f(x)&=&h(1)^{2}m(x)
 \end{array}
\qquad 
\left(x\in \mathbb{F}\right). 
\]
\item Or $\lambda_{2}\neq 0$ and then there exists a constant $\lambda\in \mathbb{K}$ such that 
 \[
  h(x)=\lambda k(x) 
  \qquad 
  \left(x\in \mathbb{F}\right). 
 \]
In this case equation \eqref{Eq_gen} is of the form 
\[
 f(xy)=k(x)y+xk(y)+\lambda^{2}k(x)k(y) 
 \qquad 
 \left(x, y\in \mathbb{F}\right). 
\]
\end{enumerate}

\begin{prop}
Let $\lambda\in \mathbb{K}$. 
Functions $f, k\colon \mathbb{F}\to \mathbb{K}$ fulfill the functional equation 
 \[
  f(xy)=xk(y)+k(x)y+\lambda^{2}k(x)k(y)
  \qquad 
  \left(x, y\in \mathbb{F}\right)
 \]
if and only if 
\begin{enumerate}[A)]
\item in case $\lambda=0$ and there exists a Leibniz function $\delta\colon \mathbb{F}\to \mathbb{K}$ such that 
\[
\begin{array}{rcl}
 k(x)&=&k(1)x+\delta(x)\\
 f(x)&=&2k(1)x+\delta(x)
\end{array}
 \qquad 
 \left(x\in \mathbb{F}\right), 
\]
\item in case $\lambda\neq 0$
\begin{enumerate}[(a)]
\item either there exists a constant $\gamma \in \mathbb{K}$ such that 
\[
 \begin{array}{rcl}
  k(x)&=&\gamma x\\
  f(x)&=&\left(\lambda^{2}\gamma^{2}+2\gamma \right) x
 \end{array}
\qquad 
\left(x\in \mathbb{F}\right). 
\]
\item or there exists a multiplicative function $m\colon \mathbb{F}\to \mathbb{K}$ and a constant 
$\gamma\in \mathbb{K}$ such that 
\[
 \begin{array}{rcl}
k(x)&=& -\dfrac{1}{\lambda^{2}}x+\gamma m(x)\\[2mm]
f(x)&=&-\dfrac{1}{\lambda^2}x+\gamma^{2}\lambda^{2}m(x)
 \end{array}
 \qquad 
 \left(x\in \mathbb{F}\right). 
\]

\end{enumerate}
\end{enumerate}
\end{prop}

\begin{proof}
 Observe, that our equation with $y=1$ immediately yields that 
 \[
  f(x)= (1+\lambda^{2}k(1))k(x)+k(1)x
  \qquad
  \left(x\in \mathbb{F}\right). 
 \]
If $\lambda=0$, then from this we get that the mapping $\tilde{k}$ defined on $\mathbb{F}$ by 
\[
 \tilde{k}(x)=k(x)-k(1)x 
 \qquad 
 \left(x\in \mathbb{F}\right), 
\]
is a Leibniz function. 

If $\lambda\neq 0$, then define the function $h\colon \mathbb{F}\to \mathbb{K}$ by 
\[
 h(x)= x+\dfrac{\lambda^{2}}{2}k(x)
 \qquad 
 \left(x\in \mathbb{F}\right)
\]
to derive 
\[
 f(xy)= k(x)h(y)+h(x)k(y) 
 \qquad 
 \left(x, y\in \mathbb{F}\right), 
\]
that is the same sine-type equation as in the proof of Proposition \ref{Prop4}. 
Similarly as there, a careful adaptation shows that alternative (a) corresponds to the case when 
$k$ and $h$ are linearly dependent, while alternative (b) corresponds to the case when 
$k$ and $h$ are linearly independent. 
\end{proof}

\begin{rem}
 If we would like to describe the additive solutions of the above functional equation, then the mapping 
 $\delta$ should be replaced by a derivation $d\colon \mathbb{F}\to \mathbb{K}$, and the mapping 
 $m$ should be replaced by a homomorphism $\varphi\colon \mathbb{F}\to\mathbb{K}$. 
\end{rem}

\subsubsection*{The non-degenerate case}

In this subsection we investigate the so-called non-degenerate case. 
More precisely, the problem to be studied is the following. 
Let \emph{$\mathbb{K}$ be an algebraically closed field with $\mathrm{char}(\mathbb{K})\neq 2$ zero and 
$\mathbb{F}$ be a subfield of $\mathbb{K}$}, 
and $f, h, k\colon \mathbb{F}\to \mathbb{K}$ be functions so that the system 
$\left\{\mathrm{id}, h, k\right\}$ is \emph{linearly independent}. 
In what follows we determine the solutions of the functional equation 
\[
 f(xy)=h(x)h(y)+xk(y)+k(x)y 
 \qquad 
 \left(x, y\in \mathbb{F}\right). 
\]

Using the results of Székelyhidi \cite{Sze91} and McKiernan \cite{McK77} delineated in the first section, 
we derive immediately that the solutions $f, h, k\colon \mathbb{F}\to \mathbb{K}$ of the above equation 
are exponential polynomials of degree at most two. 
Depending on this degree we have three different possibilities, see pages 89-92 of \cite{Sze91} where the description 
of the functional equation 
\[
 f(xy)= g_{1}(x)h_{1}(y)+g_{2}(x)h_{2}(y)+g_{3}(x)h_{3}(y) 
\]
can be found. This obviously covers our equation with the choice 
\[
 \begin{array}{rcl}
  g_{1}(x)=h_{1}(x)&=&h(x)\\
  g_{2}(x)&=&x\\
  h_{2}(x)&=&k(x)\\
  g_{3}(x)&=&k(x)\\
  h_{3}(x)&=&x
 \end{array}
\qquad 
\left(x\in \mathbb{F}\right). 
\]

The first possibility is that there exist different nonzero multiplicative functions 
$m_{1}, m_{2}, m_{3} \colon \mathbb{F}\to \mathbb{K}$ and constants $\alpha_{i}, \beta_{j}^{(i)}, \gamma_{i}^{(j)}\in \mathbb{K}$, 
$i, j=1, 2, 3$ such that 
\[
 \begin{array}{rcl}
  f(x)&=&\alpha_{1}m_{1}(x)+\alpha_{2}m_{2}(x)+\alpha_{3}m_{3}(x)\\[2mm]
  g_{i}(x)&=&\beta^{(i)}_{1}m_{1}(x)+\beta^{(i)}_{2}m_{2}(x)+\beta^{(i)}_{3}m_{3}(x)\\[2mm]
  h_{i}(x)&=&\gamma^{(i)}_{1}m_{1}(x)+\gamma^{(i)}_{2}m_{2}(x)+\gamma^{(i)}_{3}m_{3}(x)\\
 \end{array}
\qquad 
\left(x\in \mathbb{F}\right). 
\]
Condition $g_{1}=h_{1}$ implies that 
\[
 \beta_{j}^{(1)}= \gamma_{j}^{(1)} 
 \qquad 
 \left(j=1, 2, 3\right). 
\]
Similarly, from $h_{2}=g_{3}$ we obtain that 
\[
 \beta^{(3)}_{j} =\gamma_{j}^{(2)}
  \qquad 
 \left(j=1, 2, 3\right). 
\]
Finally, from
\[
 g_{2}(x)=h_{3}(x)=x 
 \qquad
 \left(x\in \mathbb{F}\right)
\]
we derive that one of the multiplicative functions $m_{1}, m_{2}, m_{3}$ is the identity, say $m_{1}$ and 
\[
 \begin{array}{rcl}
  \beta_{2}^{(2)}=\beta_{3}^{(2)}=0& \; & \beta_{1}^{(2)}=1\\[2mm]
  \gamma_{2}^{(3)}=\gamma_{3}^{(3)}=0& \; & \gamma_{1}^{(3)}=1. 
 \end{array}
\]
Using this and our functional equation, we get that for the above constants 
\[
 \begin{pmatrix}
  \beta_{1}^{(1)} &  1 & \beta_{1}^{(3)}\\
  \beta_{2}^{(1)} & 0& \beta_{2}^{(3)}\\
  \beta_{3}^{(1)} & 0 & \beta_{3}^{(3)}
 \end{pmatrix}
 \cdot 
 \begin{pmatrix}
  \beta_{1}^{(1)} &  \beta_{2}^{(1)} & \beta_{3}^{(1)}\\
  \beta_{1}^{(3)} & \beta_{2}^{(3)}& \beta_{3}^{(3)}\\
  1 & 0 & 0
 \end{pmatrix}
 =
 \begin{pmatrix}
  \alpha_{1} & 0& 0\\
  0& \alpha_{2} & 0\\
  0& 0& \alpha_{3}
 \end{pmatrix}
\]
has to hold. Especially, $\beta_{2}^{(1)}\cdot\beta_{3}^{(1)}=0$, which yields that the coefficient of $m_{2}$ or that of $m_{3}$ is zero. 
All in all, there exists a multiplicative function $m\colon \mathbb{F}\to \mathbb{K}$ and constants such that 
\[
 \begin{array}{rcl}
  f(x)&=&(\beta_{1}^{2}+2\gamma_{1})x+\beta_{2}^{2}m(x)\\
  h(x)&=&\beta_{1}x+\beta_{2}m(x)\\
  k(x)&=&\gamma_{1}x+\gamma_{2}m(x)
 \end{array}
\qquad 
\left(x\in \mathbb{F}\right). 
\]
In this case however the functions $\mathrm{id}, h$ and $k$ span a linear space of dimension at most two. 
Notice that we are interested in the case $\dim \mathrm{lin}\left(\mathrm{id}, h, k\right)=3$. Therefore the above type of solutions 
does not appear. 

The second possibility is that there exist different multiplicative functions 
$m_{1}, m_{2}$ and a logarithmic function $l\colon \mathbb{F}^{\times}\to \mathbb{K}$ and constants 
such that 
\[
 \begin{array}{rcl}
  f(x)&=&\left(\alpha_{1}l(x)+\alpha_{2}\right)m_{1}(x)+\alpha_{3}m_{2}(x)\\
  g_{i}&=&\left(\beta_{1}^{(i)}l(x)+\beta_{2}^{(i)}\right)m_{1}(x)+\beta_{3}m_{2}(x)\\
  h_{i}&=&\left(\gamma_{1}^{(i)}l(x)+\gamma_{2}^{(i)}\right)m_{1}(x)+\gamma_{3}m_{2}(x)\\
 \end{array}
\qquad 
\left(x\in \mathbb{F}, i=1, 2, 3\right). 
\]
In our case 
\[
 \begin{array}{rcl}
  g_{1}(x)=h_{1}(x)&=&h(x)\\
  g_{2}(x)&=&x\\
  h_{2}(x)&=&k(x)\\
  g_{3}(x)&=&k(x)\\
  h_{3}(x)&=&x
 \end{array}
\qquad 
\left(x\in \mathbb{F}\right), 
\]
thus either $m_{1}$ or $m_{2}$ is the identity. 
If we would have 
\[
 m_{2}(x)=x 
 \qquad 
 \left(x\in \mathbb{F}\right), 
\]
then after comparing the coefficients 
\[
 \begin{array}{rcl}
  f(x)&=&\alpha_{1}x+\alpha_{2}m(x)\\
  h(x)&=&\beta_{1}x+\beta_{2}m(x)\\
  k(x)&=&\gamma_{1}x+\gamma_{2}m(x)
 \end{array}
\qquad 
\left(x\in \mathbb{F}\right). 
\]
would follow with certain constants. Similarly as above, in this case we would have 
$\dim \mathrm{lin}\left(\mathrm{id}, h, k\right)\leq 2$, contrary to our assumptions. 

The fact that 
\[
 m_{1}(x)=x 
 \qquad 
 \left(x\in \mathbb{F}\right), 
\]
means that there exists a multiplicative function
$m\colon \mathbb{F}\to \mathbb{K}$ and a logarithmic function $l\colon \mathbb{F}^{\times}\to \mathbb{K}$ and constants 
such that 
\[
 \begin{array}{rcl}
  f(x)&=&\left(\alpha_{1}l(x)+\alpha_{2}\right)x+\alpha_{3}m(x)\\[2mm]
  h(x)&=&\left(\beta_{1}l(x)+\beta_{2}\right)x+\beta_{3}m(x)\\[2mm]
  k(x)&=&\left(\gamma_{1}l(x)+\gamma_{2}\right)x+\gamma_{3}m(x)\\
 \end{array}
\qquad 
\left(x\in \mathbb{F}\right). 
\]
Inserting this back into our equation 
the system of equations 
\[
 \begin{array}{rcl}
  \beta_{1}&=&0 \\[2mm]
  \alpha_{1}&=&\gamma_{1}\\[2mm]
  \alpha_{2}&=& \beta_{2}^{2}+2\gamma_{2}\\[2mm]
  \alpha_{3}&=&\beta_{3}^{2}
 \end{array}
\]
follow, that is, 
\[
 \begin{array}{rcl}
  f(x)&=&\left(\gamma_{1}l(x)+\beta_{2}^{2}+2\gamma_{2}\right)x+\beta_{3}^{2}m(x)\\[2mm]
  h(x)&=&\beta_{2}x+\beta_{3}m(x)\\[2mm]
  k(x)&=&\left(\gamma_{1}l(x)+\gamma_{2}\right)x+\gamma_{3}m(x)\\
 \end{array}, 
\]
where additionally $\gamma_{3}+\beta_{2}\beta_{3}=0$ also has to hold. 
To guarantee the system $\left\{\mathrm{id}, h, k\right\}$ to be linearly independent, 
$\beta_{3}\neq 0$ and $\gamma_{1}\neq 0$ also has to be supposed.

The last possibility is that all the involved functions are exponential polynomials of degree two. Since 
\[
 g_{2}(x)=h_{3}(x)= x
 \qquad 
 \left(x\in \mathbb{F}\right), 
\]
the corresponding multiplicative function is the identity, that is, we have 
\[
 \begin{array}{rcl}
  f(x)=\displaystyle \sum_{p, q=1}^{2}\alpha_{p, q}l_{p}(x)l_{q}(x)x+\displaystyle \sum_{p=1}^{2}\alpha_{p}l_{p}(x)x+\alpha x\\[2mm]
  h(x)=\displaystyle \sum_{p, q=1}^{2}\beta_{p, q}l_{p}(x)l_{q}(x)x+\displaystyle \sum_{p=1}^{2}\beta_{p}l_{p}(x)x+\beta x\\[2mm]
  k(x)=\displaystyle \sum_{p, q=1}^{2}\gamma_{p, q}l_{p}(x)l_{q}(x)x+\displaystyle \sum_{p=1}^{2}\gamma_{p}l_{p}(x)x+\gamma x\\
 \end{array}
\qquad 
\left(x\in \mathbb{F}\right)
\]
with certain constants and with linearly independent logarithmic functions $l_{1}, l_{2}\colon \mathbb{F}^{\times}\to\mathbb{K}$. 
Substituting these representations into our equation, firstly 
\[
 \alpha_{p, q}=\beta_{p, q}=\gamma_{p, q}=0 
 \qquad 
 \left(p, q\in \left\{1, 2\right\}\right)
\]
can be concluded, that is, in fact we have that 
\[
 \begin{array}{rcl}
  f(x)=\displaystyle \sum_{p=1}^{2}\alpha_{p}l_{p}(x)x+\alpha x\\[2mm]
  h(x)=\displaystyle \sum_{p=1}^{2}\beta_{p}l_{p}(x)x+\beta x\\[2mm]
  k(x)=\displaystyle \sum_{p=1}^{2}\gamma_{p}l_{p}(x)x+\gamma x\\
 \end{array}
\qquad 
\left(x\in \mathbb{F}\right). 
\]
Again, the comparison of the coefficients leads to 
\[
  \beta_{1}=0\qquad 
  \beta_{2}=0\qquad
  \alpha_{1}=\gamma_{1}\qquad
  \alpha_{2}=\gamma_{2}\qquad
  \alpha=\beta^{2}+2\gamma, 
 \]
that is, there exist linearly independent logarithmic functions $l_{1}, l_{2}\colon \mathbb{F}^{\times}\to \mathbb{K}$ and 
constants $\gamma_{1}, \gamma_{2}, \beta, \gamma \in \mathbb{K}$ 
such that 
\[
 \begin{array}{rcl}
  f(x)&=&\gamma_{1}l_{1}(x)x+\gamma_{2}l_{2}(x)x+ (\beta^{2}+2\gamma) x\\[2mm]
  h(x)&=&\beta x\\[2mm]
  k(x)&=&\gamma_{1}l_{1}(x)x+\gamma_{2}l_{2}(x)x+\gamma x
 \end{array}
 \qquad 
 \left(x\in \mathbb{F}\right). 
\]
Observe that in this case $\mathrm{id}$ and $h$ are linearly dependent, yielding that this possibility cannot occur in our situation. 
Summing up, the following statement holds true.

\begin{thm}\label{Thm5}
Let $f, h, k\colon \mathbb{F}\to \mathbb{K}$ be functions such that 
the system $\left\{\mathrm{id}, h, k\right\}$ is \emph{linearly independent}. 
The functional equation 
 \begin{equation}\label{Eq_gen_indep}
  f(xy)=h(x)h(y)+xk(y)+k(x)y 
 \end{equation}
is fulfilled for any $x, y\in \mathbb{F}$ if and only if 
there exists a multiplicative function
$m\colon \mathbb{F}\to \mathbb{K}$ and a logarithmic function $l\colon \mathbb{F}^{\times}\to \mathbb{K}$ and constants 
$\beta_{2}, \beta_{3}, \gamma_{1}, \gamma_{2}, \gamma_{3}\in \mathbb{K}$
such that 
\[
 \begin{array}{rcl}
  f(x)&=&\left(\gamma_{1}l(x)+\beta_{2}^{2}+2\gamma_{2}\right)x+\beta_{3}^{2}m(x)\\[2mm]
  h(x)&=&\beta_{2}x+\beta_{3}m(x)\\[2mm]
  k(x)&=&\left(\gamma_{1}l(x)+\gamma_{2}\right)x+\gamma_{3}m(x)\\
 \end{array}
 \qquad 
 \left(x\in \mathbb{F}\right), 
\]
where additionally $\gamma_{3}+\beta_{2}\beta_{3}=0$, $\beta_{3}\neq 0$ and $\gamma_{1}\neq 0$ also have to hold.
\end{thm}

\section{Further interpretations and open questions}

The primary aim of this paper was to study (different) notions of homo-derivations on fields 
(with and  without additivity supposition) as well as the Pexiderized version of this definition.  
At the same time, the results obtained here can be restated with the aid of the notion of alien functional equations. 
This concept was developed by J.~Dhombres in the paper \cite{Dho88}\index{Dhombres, J.}. 
However, the interested reader should also consult the survey paper Ger--Sablik \cite{GerSab17}. 

Let $X$ and $Y$ be nonempty sets and $E_{1}(f)=0$ and $E_{2}(f)=0$ be two functional equations for the function $f\colon X\to Y$. 
We say that equations $E_{1}$ and $E_{2}$ are \emph{alien}, if any solution $f\colon X\to Y$ of the functional equation 
\[
 E_{1}(f)+E_{2}(f)=0
\]
also solves the system 
\[
 \begin{array}{rcl}
  E_{1}(f)&=&0\\
  E_{2}(f)&=&0. 
 \end{array}
\]

Furthermore, equations $E_{1}$ and $E_{2}$ are \emph{strongly alien}, if any pair $f, g\colon X\to Y$ of functions that solves 
\[
 E_{1}(f)+E_{1}(g)=0
\]
also yields a solution for 
\[
  \begin{array}{rcl}
  E_{1}(f)&=&0\\
  E_{2}(g)&=&0. 
 \end{array}
\]

The following statement shows that the multiplicative Cauchy equation and the Leibniz rule equation are \emph{not strongly alien}, 
this is an easy consequence of Theorem \ref{Thm5}. 

\begin{prop}
Let $h, k\colon \mathbb{F}\to \mathbb{K}$ be functions such that the system
$\left\{\mathrm{id}, h, k\right\}$ is \emph{linearly independent}. 
The functional equation 
 \begin{equation}\label{Eq_gen_alien}
  h(xy)+k(xy)=h(x)h(y)+xk(y)+k(x)y 
 \end{equation}
is fulfilled for any $x, y\in \mathbb{F}$ if and only if 
there exists a multiplicative function
$m\colon \mathbb{F}\to \mathbb{K}$ and a logarithmic function $l\colon \mathbb{F}^{\times}\to \mathbb{K}$ and constants 
$\beta, \gamma\in \mathbb{K}$, $\beta \neq 1$, $\gamma\neq 0$
such that 
\[
 \begin{array}{rcl}
  h(x)&=&\beta x+(1-\beta)m(x)\\[2mm]
  k(x)&=&\left(\gamma l(x)+\beta-\beta^{2}\right)x+(\beta^{2}-\beta)m(x)\\
 \end{array}
 \qquad 
 \left(x\in \mathbb{F}\right).  
\]
\end{prop}

As the proposition below shows, the multiplicative Cauchy equation and the Leibniz rule equation 
are alien, cf. Proposition \ref{Prop1} and take $\lambda=\mu$ in the corollary below. 

\begin{cor}
 Let $\lambda, \mu\in \mathbb{K}$ be arbitrarily fixed constants not vanishing simultaneously. 
 Function $f\colon \mathbb{F}\to\mathbb{K}$ fulfills the functional equation 
 \begin{equation}\label{Eq8}
  \lambda \left[f(xy)-f(x)y-xf(y)\right]+\mu\left[f(xy)-f(x)f(y)\right]=0
  \qquad 
  \left(x, y\in \mathbb{F}\right)
 \end{equation}
if and only if 
\begin{enumerate}[(A)]
 \item either $\lambda=0$ and $f$ is multiplicative;
 \item or $\mu=0$ and $f$ is a Leibniz function;
 \item or none of them is zero and 
 \begin{enumerate}
  \item $f$ is identically zero 
  \item or 
 \[
  f(x)= \dfrac{\mu-\lambda}{\mu}\cdot x 
  \qquad 
  \left(x\in \mathbb{F}\right). 
 \]
 \end{enumerate}
\end{enumerate}
\end{cor}

\begin{rem}
 Under the assumptions of the previous corollary, the additive solutions of equation \eqref{Eq8} are the following. 
 \begin{enumerate}[(A)]
 \item either $\lambda=0$ and $f$ is a homomorphism;
 \item or $\mu=0$ and $f$ is a derivation;
 \item or none of them is zero and 
 \begin{enumerate}
  \item $f$ is identically zero 
  \item or 
 \[
  f(x)= \dfrac{\mu-\lambda}{\mu}\cdot x 
  \qquad 
  \left(x\in \mathbb{F}\right). 
 \]
 \end{enumerate}
\end{enumerate}
\end{rem}

\begin{opp}
In this paper equation \eqref{Eq_homo-deriv} was considered  under rather mild assumptions on the domain and also on the range. 
At the same time, while investigating functional equation \eqref{Eq_gen} we always assumed that the range 
 of the involved mappings is a commutative, algebraically closed field $\mathbb{K}$ with $\mathrm{char}(\mathbb{K})\neq 2$. 
 The reason for this is that our main tools were Theorem \ref{Szekely} and the related results of McKiernan \cite{McK77}. 
 Clearly, the above equations can be studied without these assumptions. 
 We remark that in case of the so-called \emph{degenerate cases} a careful adaptation of the proofs shows that the same results hold true for commutative rings
  (at some places we have to assume that the range does not have any zero-divisors). 
 Therefore we can formulate the following open questions. 
 \begin{enumerate}[(A)]
  \item The algebraic closedness of the field $\mathbb{K}$ essential in our results. Nevertheless, if the field $\mathbb{K}$ is not algebraically closed, then we may
  take $\mathrm{algcl}(\mathbb{K})$ as the extended range. Using our method, $\mathrm{algcl}(\mathbb{K})$-valued solutions can be described and 
  every $\mathbb{K}$-valued solution belongs to the above larger solution space. How can these solutions be recognized in the larger solution space?   
  \item To apply the results of Székelyhidi \cite{Sze91} and McKiernan \cite{McK77}, the assumption 
  $\mathbb{F}\subset \mathbb{K}$ be a field is sufficient, but may not be necessary. 
  Observe that we only need $\mathbb{F}$ to be a (commutative) subring of $\mathbb{K}$. In case the range 
  $Q$ is only a (commutative) ring, then what else should be supposed about $Q$, to get the same results? \\
  It is worth to note that if $Q$ has no zero divisors, then up to isomorphism there exists a unique field of fractions that we may denote by 
  $\mathbb{K}$. In this case $\mathbb{K}$-valued functions can be considered and if it is needed we should take the algebraic closure 
  of this field (c.f. part (A)). 
  \item In case $P\subset Q$ are (commutative) rings and we consider mappings from $P$ to $Q$, then are there different solutions of \eqref{Eq_gen} then presented here? 
 \end{enumerate}
\end{opp}
\vspace{0.5cm}

\noindent
\emph{Acknowledgement.}
Gergely Kiss was supported by the Hungarian National Foundation for
Scientific Research, Grant No. K124749 and the Premium Postdoctoral
Fellowship of Hungarian Academy of Science.
The research of Eszter Gselmann has been carried out with the help of the project 2019-2.1.11-T\'{E}T-2019-00049,
which has been implemented with the support provided from the National Research, Development
and Innovation Fund of Hungary, financed under the T\'{E}T funding scheme.


\begin{thebibliography}{10}

\bibitem{AlhMut19a}
E.~F. Al~harfie and N.~M. Muthana.
\newblock Homoderivations of prime rings with involution.
\newblock {\em Bulletin of the International Mathematical Virtual Institute},
  9(2):305--318, 2019.

\bibitem{AlhMut19}
E.~F. Al~harfie and N.~M. Muthana.
\newblock On homoderivations and commutativity of rings.
\newblock {\em Bulletin of the International Mathematical Virtual Institute},
  9(2):301--304, 2019.

\bibitem{Dho88}
J.~Dhombres.
\newblock Relations de d\'ependance entre les \'equations fonctionnelles de
  {C}auchy.
\newblock {\em Aequationes Math.}, 35(2-3):186--212, 1988.

\bibitem{Sof00}
M.~M. El~Sofy.
\newblock Rings with some kinds of mappings.
\newblock Cairo University, Branch of Fayoum (M. Sc. Thesis), 2000.

\bibitem{GerSab17}
R.~Ger and M.~Sablik.
\newblock Alien functional equations: a selective survey of results.
\newblock In {\em Developments in functional equations and related topics},
  volume 124 of {\em Springer Optim. Appl.}, pages 107--147. Springer, Cham,
  2017.

\bibitem{Kha91}
V.~K. Kharchenko.
\newblock {\em Automorphisms and derivations of associative rings}, volume~69
  of {\em Mathematics and its Applications (Soviet Series)}.
\newblock Kluwer Academic Publishers Group, Dordrecht, 1991.
\newblock Translated from the Russian by L. Yuzina.

\bibitem{Kuc09}
M.~Kuczma.
\newblock {\em An introduction to the theory of functional equations and
  inequalities}.
\newblock Birkh\"{a}user Verlag, Basel, second edition, 2009.
\newblock Cauchy's equation and Jensen's inequality, Edited and with a preface
  by Attila Gil\'{a}nyi.

\bibitem{MahAlaNaj19}
A.~Maha, A.~Alaa, and M.~Najat.
\newblock Some commutativity theorems for rings with involution involving
  generalized derivations.
\newblock {\em Palest. J. Math.}, 8(2):169--176, 2019.

\bibitem{McK77}
M.~A. McKiernan.
\newblock Equations of the form {$H(x\circ y)=\sum _{i} f_{i}(x)g_{i}(y)$}.
\newblock {\em Aequationes Math.}, 16(1-2):51--58, 1977.

\bibitem{MehPaj03}
M.~Mehdi~Ebrahimi and H.~Pajoohesh.
\newblock Inner derivations and homo-derivations on {$l$}-rings.
\newblock {\em Acta Math. Hungar.}, 100(1-2):157--165, 2003.

\bibitem{RehAlRadMut19}
N.~Rehman, R.~M. Al-omary, and N.~M. Muthana.
\newblock A note on multiplicative (generalized) ({$\alpha$},
  {$\beta$})-derivations in prime rings.
\newblock {\em Ann. Math. Sil.}, 33(1):266--275, 2019.

\bibitem{Sze91}
L.~Sz\'{e}kelyhidi.
\newblock {\em Convolution type functional equations on topological abelian
  groups}.
\newblock World Scientific Publishing Co., Inc., Teaneck, NJ, 1991.

\bibitem{ZarSam75}
O.~Zariski and P.~Samuel.
\newblock {\em Commutative algebra. {V}ol. {II}}.
\newblock Springer-Verlag, New York-Heidelberg, 1975.
\newblock Reprint of the 1960 edition, Graduate Texts in Mathematics, Vol. 29.

\end{thebibliography}

\end{document}